\documentclass[preprint]{elsarticle}

\usepackage{amsfonts}
\usepackage{amsmath}
\usepackage{amsthm}
\usepackage{amssymb}
\usepackage[all]{xy}
\usepackage{graphicx}
\usepackage{color}
\usepackage{youngtab}
\input xy
\xyoption{all}

\newtheorem{theorem}{Theorem}[section]
\newtheorem{corollary}[theorem]{Corollary}
\newtheorem{lemma}[theorem]{Lemma}
\newtheorem{proposition}[theorem]{Proposition}
\theoremstyle{definition}
\newtheorem{definition}[theorem]{Definition}
\newtheorem{remark}[theorem]{Remark}

\newtheorem{notation}[theorem]{Notation}

\newcommand{\fe}{\operatorname{fe}}
\newcommand{\rk}{\operatorname{rk}}

\newcommand{\sol}{\operatorname{sol}}
\newcommand{\Pic}{\operatorname{Pic}}
\newcommand{\Det}{\operatorname{det}}

\journal{Linear Algebra and Its Applications}

\begin{document}

\begin{frontmatter}

\title{Partitions of elements in a monoid and its applications to systems theory\tnoteref{label1}
}
\tnotetext[label1]{This work has been partially supported by the spanish National Institute of Cyber-Security (INCIBE) accordingly to the rule 19 of The Digital Confidence Plan and the Universidad de Le\'on under the contract X43.}

\author{Miguel V. Carriegos}
\ead{miguel.carriegos@unileon.es}
\address{Instituto de Ciencias Aplicadas a Ciberseguridad, Departamento de Matem\'{a}ticas, Universidad de Le\'on}

\author{Noem\'{i} DeCastro--Garc\'{i}a\tnoteref{label2}}
\ead{ncasg@unileon.es}
\address{Departamento de Matem\'{a}ticas, Universidad de Le\'on}

\tnotetext[label2]{First communicated as the talk \emph{Enumeration of classes of feedback isomorphisms of Locally Brunovsky Linear Systems.} at the 19th. International Linear Algebra Society Meeting, Seoul, Korea, 2014.}

\begin{abstract}
The feedback class of a locally Brunovsky linear system is fully determined by the decomposition of state space as direct sum of system invariants \cite{miguel2013}. In this paper we attack the problem of enumerating all feedback classes of locally Brunovsky systems over a $n$-dimensional state space and translate to the combinatorial problem of enumerating all the partitions of integer $n$ in some abelian semigroup. The problem of computing the number $\nu(n,k)$ of all the partitions of integer $n$ into $k$ different summands is pointed out.
\end{abstract}

\begin{keyword}
Linear systems \sep feedback enumeration problem \sep partitions of integer $n$ into $k$ different summands

\MSC[2010]  93B10 \sep 15A21 \sep 05A17 \sep 13C10 
\end{keyword}
\end{frontmatter}



\section{Introduction}

It is well known \cite{B}, \cite{Kalman}, that the number of feedback equivalence classes of reachable control systems over a $n$-dimensional $\mathbb{K}$-vector space equals the number $p_{\mathbb{N}}(n)$ of partitions of integer $n$. This number equals the number, $\sol_{\mathbb{N}}$, of solutions in $\mathbb{N}$ of linear diophantine equation
\begin{equation}\label{ne+}
n=z_1+2z_2+\cdots+nz_n
\end{equation}

Above result is generalized in \cite{miguel2013} to the general framework of regular (locally Brunovsky) linear systems over a commutative ring. In fact the number of feedback equivalence classes of regular systems with (finitely generated projective) state space $X$ equals the number $\sol_{\mathbf{P}(R)}$ of solutions of the linear equation 
\begin{equation}
X=Z_1\oplus Z_2^2\oplus\cdots\oplus Z_n^n
\end{equation}
in monoid $(\mathbf{P}(R),\oplus)$ of finitely generated projective $R$-modules. This number equals the number $p_{\mathbf{P}(R)}(X)$ of partitions(direct sum decompositions) of projective module $X$ into direct summands if monoid $\mathbf{P}(R)$ happens to be cancellative.

Our goal in this paper is, applying results of \cite{miguel2013},  to give a complete account and obtain formulae relating the number of classes of feedback isomorphisms of regular systems with state space $X$ over different commutative rings with unit element using partitions.

 First, we compute that number when $R$ is a finite product of rings $R\simeq R_{1} \times \ldots \times R_{t}$ in terms of each direct factor $R_{i}$. This case is a generalization of $R=\mathbb{Z}/l\mathbb{Z}$, the ring of modular integers. Finally, we compute the case of $R$ being a Dedekind domain. 
 
The paper is organized as follows: In Section 2 we have some details about monoid $\mathbf{P}(R)$ of isomorphism classes of finitely generated projective $R$-module. 
In Section 3 we obtain the number of classes of feedback isomorphisms of regular systems with state space $X$ when the ring is projectively trivial and when the ring splits as a finite product of rings. In Section 4, we obtain the formula over Dedekind domains. Finally, we give our conclusions.

\section{The equation}

The paper deals with the solutions of equation 
\begin{equation}
\label{Xeoplus}
X=Z_1\oplus Z_2^2\oplus\cdots\oplus Z_n^n
\end{equation}
in the monoid $(\mathbf{P}(R),\oplus)$ of isomorphism classes of finitely generated projective $R$-modules. 

Let us review some elementary properties of $\mathbf{P}(R)$ which will be applied in the sequel. The reader is referred to \cite{Kbook} for more details.

\begin{proposition}
\label{preli}
Let $R$ be a commutative ring and let $\mathbf{P}(R)$ be the set of isomorphism classes of finitely generated projective $R$-modules. Then the following properties hold:
\begin{itemize}
\item[(i)] $\mathbf{P}(R)$ is a monoid under operation $[P]\oplus [Q] =[P\oplus Q]$. Identity element is the zero $R$-module.
\end{itemize}
In the sequel we denote by $P$ the finitely generated projective $R$-module and its isomorphism class.
\begin{itemize}
\item[(ii)] $\mathbf{P}(R)$ is a commutative monoid (i.e. $P \oplus Q=Q \oplus P$)
\item[(iii)] $\mathbf{P}(R)$ is a zero-sum-free monoid (i.e. $P\oplus Q = 0 \Rightarrow P=Q=0$)
\item[(iv)] The mapping $\varphi:(\mathbb{N},+)\rightarrow (\mathbf{P}(R),\oplus)$ sending $0\mapsto 0$ and $n\mapsto R^n$ is an injective morphism of monoids.
\item[(v)] If every finitely generated projective $R$-module is free (i.e. $R$ is projectively trivial) then above morphism $\varphi$ is an isomorphism.
\end{itemize}

If $R$ is a domain with field of fractions $\mathbb{K}_R$ then every finitely generated $R$-module $P$ has constant rank $\rk(P)=\dim (P\otimes_{R} \mathbb{K}_R)$. 
\begin{itemize}
\item[(vi)] Mapping $\rk:(\mathbf{P}(R),\oplus)\rightarrow (\mathbb{N},+)$ is a monoid morphism
\item[(vii)] $\rk$ is left inverse of $\varphi$; that is, $\rk\circ\varphi=Id_{\mathbb{N}}$
\end{itemize}
\end{proposition}

We also need some notation:
\begin{notation}
\label{notation}
Let $\Sigma$ be a regular linear system.

We denote 
\begin{enumerate}
\item By $fe_{R}(X)$ as the number of feedback classes of regular systems over $R$ with state space $X$. 
\item By $sol_{\mathbf{P(R)}}(X)$ as the number of solutions of equation (\ref{Xeoplus}) in the $\mathbf{P}(R)$ where $X$ is a finitely generated projective $R$-module.
\end{enumerate}
\end{notation}

Let us remark that above two numbers are equal \cite{miguel2013}. This gives the pass from systems theory to combinatorial issues.
\section{The equation in monoids $\mathbb{N}$ and $\mathbb{N}^t$. Projectively trivial rings and finite product of rings.}

Now we study the equation (\ref{Xeoplus}) in monoid of nonnegative integers and its finite products. From the systems theory point of view, this is to solving the case of regular systems over projectively trivial rings and over finite product of projectively trivial rings. 

Examples of projectively trivial rings are: Fields $\mathbb{K}$; local rings like $\mathbb{K}[[x_1,...,x_s]]$ or $\mathbb{Z}/p^r\mathbb{Z}$, $p$ prime; principal ideal domains like $\mathbb{Z}$ or $\mathbb{K}[x]$; and polynomial rings like $\mathbb{K}[x_1,...,x_s]$ or $\mathbb{Z}[x_1,...,x_s]$. Example of finite products of projectively trivial rings are modular integers rings and their rings of polinomials $(\mathbb{Z}/m\mathbb{Z})[x_{1},\ldots,x_{s}]$.

\subsection{Projectively trivial rings}

If $R$ is projectively trivial, then the number $\fe_R(R^n)$ of feedback  classes of isomorphisms of regular systems over $R^n$ (via the isomorphism $\varphi:\mathbf{P}(R)\cong\mathbb{N}$) equals the number, $\sol_{\mathbf{P}(R)}(R^n)$, of solutions of the linear equation in $\mathbb{N}$ 
\begin{equation}
n=z_1+2z_2+\cdots +nz_n
\end{equation}
This number is the number of partitions $p_{\mathbb{N}}(n)$ of integer $n$. Thus we have the result:

\begin{theorem}[cf. Corollary 8.1 \cite{miguel2013}]\label{teoprincipal}
Let $R$ be a projectively trivial ring. Then
\begin{equation}
\fe_R(R^n)=\sol_{\mathbf{P}(R)}(R^n)=\sol_{\mathbb{N}}(n)=p_{\mathbb{N}}(n)
\end{equation}
\end{theorem}

\subsection{Product rings}

Now suppose that $R=R_1\times\cdots\times R_t$ is a finite product of rings. Our goal in this section is to prove the formula
\begin{equation}
\sol_{\mathbf{P}(R)}(R^n)=\sol_{\mathbf{P}(R_1)}(R_1^n)\cdot \ldots \cdot \sol_{\mathbf{P}(R_t)}(R_t^n)
\end{equation}
and thus, from the systems theory point of view we will have the account
\begin{equation}
\fe_R(R^n)=\fe_{R_1}(R_1^n) \cdot \ldots \cdot\ fe_{R_t}(R_t^n)
\end{equation}

In order to prove above formulae we need to describe the structure of finitely generated projective $R$-modules when $R$ is a direct product of rings.

\begin{lemma}
\label{teo}
If $R\simeq R_{1}\times\cdots\times R_{t}$ is a finite product of rings.Then the following holds
\begin{center}
$\mathbf{P}(R) \cong \mathbf{P}(R_{1})\times...\times \mathbf{P}(R_{t})$.
\end{center}
\end{lemma}
\begin{proof}
The set $M(n,R)$ of $n\times n$ matrices over $R$ is embedded in  $M(n+1,R)$ by
\begin{center}
 $(a) \mapsto \left(\begin{array}{cc}
 a & 0 \\
 0& 0
 \end{array}
 \right)$
 \end{center}
and $\displaystyle{M(R)=\bigcup_{n\geq 1}M(n,R)}$. Note that every matrix in $M(R)$ has finite size. The set of idempotent matrices in $M(R)$ is denoted by $Idem(R)$.

On the other hand $GL(n,R)$ is embedded in $GL(n+1,R)$ by
\begin{center}
$a \mapsto \left(\begin{array}{cc}
 a & 0 \\
 0& 1
 \end{array}
 \right)$
 \end{center}
 and $\displaystyle{GL(R)=\bigcup_{n\geq 1}GL(n,R)}$. Every matrix in $GL(R)$ is invertible having finite size. 

$\mathbf{P}(R)$ may be identified  \cite[Th. 1.2.3]{Rosenberg} with the set of conjugation orbits of group $GL(R)$ on set $Idem(R)$. Since $GL(R)\cong GL(R_1)\times\cdots\times GL(R_t)$ acts on $Idem(R)=Idem(R_1)\times\cdots\times Idem(R_t)$ componentwise it follows the result.

\end{proof}

\begin{theorem}
Let $R\simeq R_{1}\times \ldots \times R_{t}$ be a finite product of rings. Then:
\begin{itemize}
\item[(i)] $\sol_{\mathbf{P}(R)}(R^n)=\sol_{\mathbf{P}(R_1)}(R_1^n)\cdot \ldots \cdot \sol_{\mathbf{P}(R_t)}(R_t^n)$
\item[(ii)] $fe_{R}(R^n)=fe_{R_{1}}(R_1^n) \cdot \ldots \cdot fe_{R_{t}}(R_t^n)$
\end{itemize}
\end{theorem}
\begin{proof}
By Lemma \ref{teo},  $\mathbf{P}(R) \cong \mathbf{P}(R_{1})\times...\times \mathbf{P}(R_{t})$ we can solve the equation $R^n \simeq Z_1\oplus Z_2^2\oplus\cdots\oplus Z_n^n$ componentwise on every factor ring. Therefore the conditions hold.
\end{proof}

\begin{corollary}
\label{coro}
Let $R \simeq R_{1} \times \ldots \times R_{t}$ be a finite product of projectively trivial rings $R_{i}$ for $i=1, \ldots t$. Then $\fe_{R}(R^n)=\left(\sol_{\mathbb{N}}(n)\right)^{t}=\left(p_{\mathbb{N}}(n)\right)^t$ 
 \end{corollary}
\begin{proof}
Since the rings $R_{i}$ are  projectively trivial rings, then all finitely generated projective modules over each $R_{i}$ are free, that is, $\mathbf{P}(R_{i})\simeq \mathbb{N}$ for $i=1,...,t$.

By Theorem \ref{teo} we deduce that 
\begin{equation}
\label{bu}
\mathbf{P}(R)\simeq \mathbf{P}(R_{1})\times ... \times \mathbf{P}(R_{t})\simeq \mathbb{N} \times\overset{t)}{...} \times \mathbb{N}.
\end{equation}

So, the number of classes of feedback isomorphisms of locally Brunovsky linear systems over $R$ is the number of solutions of the equation
\begin{center}
$m = x_{1}+2x_{2}+3x_{3}+...+sx_{s}$ where $m$ and $x_{i}$ are t-uples of natural numbers.
\end{center}

Now, $m=(m_{1}, \ldots, m_{t}) \in \mathbf{P}(X)$ and $X\simeq R^{n}$  then $m=(n, \overset{t)}{\ldots},n)$. So, the equation to solve is
\begin{center}
$(n,n,\overset{t)}{...},n)=(a_{1},b_{1},...,t_{1})+2(a_{2},b_{2},...,t_{2})+3(a_{3},b_{3},...,t_{3})+...+s(a_{s},b_{s},...,t_{s})$,
\end{center}

Then, we look for the number of solutions of the system
\begin{equation}
\left\{\begin{array}{l}
n= a_{1}+2a_{2}+3a_{3}+...+sa_{s} \\
n=b_{1}+2b_{2}+3b_{3}+...+sb_{s} \\
\vdots \\
n=t_{1}+2t_{2}+3t_{3}+...+st_{s}\\
\end{array}\right.
\end{equation}
where $n, a_{i}, b_{i},...,t_{i} \in \mathbb{N}$ for $i=1,...,t$.

Now, since $\mathbb{N}$ is cancellative, the number of solutions of i-th equation is equal to  $p(n)$, the partitions of $n$ for each $i=1,...t$. Thus, the number of solutions of the system is equal to 
\begin{center}
$fe_{R}(R^{n})=p_{\mathbb{N}}(n)\cdot p_{\mathbb{N}}(n) \cdot \overset{t)}{\ldots} \cdot p_{\mathbb{N}}(n)= (p_{\mathbb{N}}(n))^{t}$
\end{center}
\end{proof}
\begin{corollary}
For the ring $R= \mathbb{Z}_{l} \simeq \mathbb{Z}_{p_{1}^{r_{1}}} \times \ldots \times \mathbb{Z}_{p_{t}^{r{t}}}$ we have
\begin{center}
$fe_{\mathbb{Z}_{l}}(R^{n})= (p_{\mathbb{N}}(n))^{t}$
\end{center}
\end{corollary}
\begin{proof}
Since $\mathbf{P}(\mathbb{Z}/{p_{i}^{r_{i}}}\mathbb{Z})\simeq \mathbb{N}$ for each $i$, we conclude the proof. 
\end{proof}

\section{The equation in $\mathbb{N}\times G$, G abelian group. Dedekind domains.}

Let $R$ be a commutative ring.  Let $\mathrm{Pic}(R)$ be the set of isomorphism classes of line bundles over $R$ (finitely generated projective $R$-modules of $rank$ one). Then  \cite[I\S3]{Kbook} $(\mathrm{Pic}(R),\otimes_R)$ is an abelian group where $R=1_{\mathrm{Pic}(R)}$ and $P^{-1}=\mathrm{Hom_R(P,R)}$. 

In the sequel $R$ denotes a Dedekind domain with field of fractions $\mathbb{K}(R)$; that's to say, a commutative domain (no nonzero zero-divisors) which is noetherian, integrally closed and $1$-dimensional. Then \cite[I.3.4]{Kbook}  finitely generated projective $R$-module $P$ is completely classifed by its rank $\rk(P)=\dim (P \otimes \mathbb{K}(R)$) and its determinant $\displaystyle{\wedge^{\mathrm{rk}(P)}P}$. To be precise, $P$ is isomorphic to $R^{\mathrm{rk} (P)-1}\oplus \displaystyle{\wedge^{\mathrm{rk}(P)} P}$.

Thus, $\mathbf{P}(R)$ equals $[\mathbb{N}^{+} \times Pic(R)]\cup \{0\}$ as a set. Arithmetic in $(\mathbf{P}(R), \oplus)$ is given by

\begin{equation}
P\oplus Q \cong R^{\rk(P)+\rk(Q)-1}\oplus \left( \displaystyle{\wedge^{\mathrm{rk}(P)} P}\otimes_{R} \displaystyle{\wedge^{\mathrm{rk}(Q)} Q}\right)
\end{equation}
where zero module $0$ is the identity and internal law $\oplus$ is commutative by Proposition \ref{preli}.
\begin{notation}
The determinant line bundle $\wedge^{\mathrm{rk}(P)}$ of a finitely generated $R$- module $P$ is denoted by $\Det(P)$.
\end{notation}

\begin{remark}
\label{remark}
Let $X$ be a finitely generated projective $R$-module. Looking for the number of classes of feedback isomorphisms of regular systems over  $X$ is equivalent to computing the number of solutions of equation (\ref{Xeoplus})
\begin{equation}
X=Z_1\oplus Z_2^2\oplus\cdots\oplus Z_n^n
\end{equation}
in $\mathbf{P}(R)\cong [\mathbb{N}^{+}\times Pic(R)]\cup \{0\}$. These solutions are determined by solutions of
\begin{equation}
\label{rank}
\rk(X)=\rk(Z_1)+2\rk(Z_2)+\cdots+n\rk(Z_n)
\end{equation}
in $(\mathbb{N}, +)$ together with a solution of
\begin{equation}
\label{det}
\Det(X)=\Det(Z_1)\otimes\Det(Z_2)^{\otimes 2}\otimes\cdots\otimes\Det(Z_n)^{\otimes n}
\end{equation}
in $Pic(R)$.

Note that if we only fix the rank of the state space, $rk(X)=n$, then $X\cong R^{n-1}\oplus L$ and 
\begin{equation*}
\Det(X)= \Det(R^{n-1}\oplus L)= \displaystyle{\wedge^{n} (R^{n-1}\oplus L)}=\displaystyle{\bigoplus^{n}_{i=0}[(\wedge^{i} R^{n-1})\otimes (\wedge^{n-i} L)]}=
\end{equation*}
\begin{equation*}
=\displaystyle{\bigoplus^{n-2}_{i=0} [R^{\binom{n-1}{i}}\otimes 0] \oplus[(\wedge^{n-1} R^{n-1})\otimes (\wedge^{1} L)] \oplus [\wedge^{n} R^{n-1} \otimes \wedge^{0} L]}= R\otimes L = L 
\end{equation*}
and equation (\ref{det}) turns to be
\begin{equation}
L=\Det(Z_1)\otimes\Det(Z_2)^{\otimes 2}\otimes\cdots\otimes\Det(Z_n)^{\otimes n}
\end{equation}
\end{remark}
\begin{remark}
\label{detcero}
Let $P$ a finitely generated projective module over a Dedekind domain $R$. In particular note that if $\rk(P)=0$, then $\Det(P)=R$.

Then, note that the solutions of equation (\ref{det}) in $Pic(R)$ are entangled with solutions of ranks equation in $\mathbb{N}$.
\end{remark}
\begin{remark}
\label{alfa}
A classical result by Claborn \cite{Claborn} shows that given any abelian group $G$, there exists a Dedekind domain $R$ such that $Pic(R) \cong G$.  

Let $\mid Pic(R) \mid=p$ be the order of $Pic(R)$ with $p$ prime. We will use the isomorphism

\begin{center}
$(Pic(R)$, $\otimes) \overset{\overset{\alpha}{\cong}}{\longrightarrow}$ $(\mathbb{Z}/p\mathbb{Z}, +)$
\end{center}
\begin{center}
$L$ \hspace{0.5cm} $ \mapsto$ \hspace{0.3cm} $ \alpha(L)$
\end{center}
where in particular $\alpha(R)=\overline{0}$ in  $\mathbb{Z}/p\mathbb{Z}$
\end{remark}

Finally, we introduce some notation about the account of regular systems over a Dedekind domain, $R$.
\begin{notation}
\begin{enumerate}
\item $fe_{R}(X)$ denotes the number of feedback classes of regular systems over $X$ (\ref{notation}). 
\item $fe_{R}(n)$ will be the number of classes of feedback isomorphisms of regular systems over $R$ with a state space a finitely projective $R$-module of rank $n$.
\item In the case of Dedekind domains $X\cong R^{n-1}\oplus L$ and therefore $fe_{R}(X)= fe_{R}(R^{n-1}\oplus L )$. 
\end{enumerate}
\end{notation}

In order to solve above equations (\ref{rank}) and (\ref{det}) we introduce the following combinatorial number 
\begin{definition}
Let $n$ be a positive integer and $1\leq k \leq n$. We denote by $\nu(n,k)$ the set of partitions of integer $n$ into $k$ different summands. We also denote by $\nu(n,k)$ its cardinal.
\end{definition}

As matter of example $\nu(6,2)$ is the number of partitions of integer $6$ into $2$ different summands and hence contains exactly partitions $$(51),(42),(411),(3111),(2211),(21111)$$ and therefore $\nu(6,2)=6$.

\begin{definition}
\label{nukp}
Let $p$ be a prime number. We denote by $\nu(n,k,p)$ the set of partitions in $\nu(n,k)$ where all coefficientes of the summands are multiples of $p$. We also denote by $\nu(n,k,p)$ its cardinal.

For convenience let's denote by $\nu'(n,k,p)= \nu(n,k)-\nu(n,k,p)$.
\end{definition}
 As matter of example, $\nu(6,2,2)=1$ because the only partition in $\nu(6,2)$ with the property that all summands are multiple of $2$ is $(42)$ .

\begin{remark}
Combinatorial number $\nu(n,k)$ needs further study. We only point out two straightforward properties:
\begin{itemize}
\item[(i)] $\nu(n,1)=\text{div}(n)$; that is $\nu(n,1)$ equals the number of divisors (including both $1$ and $n$) of integer $n$
\item[(ii)] If $n<k(k+1)/2$ then $\nu(n,k)=0$ because the least partition one can form with $k$ different summands is $(k,k-1,...,2,1)$ and therefore an $n\geq 1+2+\cdots +k=\frac{k(k+1)}{2}$ is needed.
\end{itemize}
\end{remark}

We state our main result:

\begin{theorem}
Let $R$ be a Dedekind domain and let $Pic(R)$ be its Picard Group. Then, the number of feedback classes of regular systems is as follows:
\begin{itemize}
\item[(i)] $fe_{R}(n)$ is the number of solutions $(Z_{1},Z_{2},\ldots,Z_{n})$ of equation
\begin{equation}
\label{rankn}
n=rk(Z_{1})+2rk(Z_{2})+\ldots+nrk(Z_{n}) \textit{ in } (\mathbb{N}, +)
\end{equation}
\item[(ii)] If $\mid Pic(R) \mid=\infty$ then $\fe_R(n)=\infty$.
\item[(iii)] If $\mid Pic(R) \mid=d<\infty$ then $fe_R(n)=\displaystyle{\sum_{k=1}^{n}\nu(n,k)\cdot d^k}$
\item[(iv)] $fe_{R}(X)$ is the number of solutions $(Z_{1},Z_{2},\ldots,Z_{n})$ of the system of equations (see Remarks \ref{remark} and \ref{detcero}).
\begin{center}
$\left\{\begin{array}{l}
\rk(X)=\rk(Z_1)+2\rk(Z_2)+\cdots+n\rk(Z_n)$ in $(\mathbb{N}, +) \\
\\
\Det(X)=L=\Det(Z_1)\otimes\Det(Z_2)^{\otimes 2}\otimes\cdots\otimes\Det(Z_n)^{\otimes n}$ in $(Pic(R),\otimes).\\
\end{array}\right.$
\end{center}

\item[(v)] If $\mid Pic(R) \mid=p$ is prime then $fe_R(X\simeq R^{n})=\displaystyle{\sum_{k=1}^{n}\left[\nu(n,k,p)\cdot p^{k}+\nu'(n,k,p)\cdot p^{k-1}\right]}$.
\item[(vi)])  If $\mid Pic(R) \mid=p$ is prime then $fe_{R}(R^{n-1}\oplus L)=\displaystyle{\sum_{k=1}^{n}\nu'(n,k,p)\cdot p^{k-1}}$

\end{itemize}
\end{theorem}

\begin{proof}
\begin{enumerate}
\item[(i)] Is clear by Remark \ref{remark}.
\item[(ii)] Suppose that $Pic(R)$ is of infinite order and $L$ varies in $Pic(R)$. Then 
\begin{center}
$(Z_1=R^{n-1}\oplus L,Z_2=0,...)$ 
\end{center}
are infinitely many different solutions of equation (\ref{rankn}).

\item[(iii)] $\nu(n,k)$ is the set of solutions $(\rk(Z_{1}),\rk(Z_{2}),\ldots,\rk(Z_{n}))$ of the equation (\ref{rankn}) where $k$ of the entries of above tuple are non zero. Thus 
\begin{center}
$(\rk(Z_{1}),\rk(Z_{2}),\ldots,\rk(Z_{n}))=(0,\ldots,\rk(Z_{i_{1}}),\ldots,\rk( Z_{i_{k}}),\ldots,0,\ldots)$
\end{center}
In order to realize solutions $(Z_{1}, \ldots, Z_{n})$ we are free to choose $L_{1},\ldots, L_{k}$ in $Pic(R)$ to obatin solutions
\begin{center}
$(0,\ldots,0,R^{rk(Z_{i_{1}})-1}\oplus L_{1},0,\ldots,0, R^{rk(Z_{i_{k}})-1}\oplus L_{k},0,\ldots)$
\end{center}
Since $L_{i}$ varies in $Pic(R)$, then there are exactly $d^{k}$ different choices and therefore
\begin{center}
$fe_R(n)=\displaystyle{\sum_{k=1}^{n}\nu(n,k)\cdot d^k}$
\end{center}
\item[(iv)] Is clear from the Remark \ref{remark}.

\item[(v)] If $\mid Pic(R) \mid=p$ is prime, then by Remark \ref{alfa}, $\Pic(R) \cong \mathbb{Z}/p\mathbb{Z}$, and equations giving $fe_{R}(R^{n})$ are:
\begin{center}
$\left\{\begin{array}{l}
n=\rk(Z_1)+2\rk(Z_2)+\cdots+n\rk(Z_n)$ in $(\mathbb{N}, +) \\
\\
0=a_{1} + 2a_{2} + \cdots + na_{n}$ in $(\mathbb{Z}/p\mathbb{Z},+).\\
\end{array}\right.$
\end{center}
where $a_{i}=\alpha(\Det(Z_{i}))$

There are exactly $\nu(n,k)$ different solutions for the ranks equation with exactly $k$ non zero $\rk(Z_{i})'s$. Every solution of ranks equation gives some choices for the second equation. But is crucial to know how many coefficients are non zero modulo $p$.

The equation over determinants by $\alpha$ is on the form
\begin{equation}
0= a_{1}+2a_{2}+\ldots+pa_{p}+...+(2p)a_{2p}+\ldots+ na_{n} \textit{ in } \mathbb{Z}/p\mathbb{Z}
\end{equation}
Let us reorder the summands such that we have $l= \lfloor \frac{n}{p}\rfloor$ summands which coefficientes are multiple of $p$, and $n-l$ summands whose coefficientes are prime with $p$.
\begin{equation*}
\label{ecup}
0= \overbrace{pa_{p}+2pa_{2p}+\ldots+lpa_{lp}}^{l \text{ summands}}+
\overbrace{a_{1}+\ldots+(p-1)a_{p-1}+(p+1)a_{p+1}+...}^{(n-l) \text{ summands }}
\end{equation*}

Since the group of $l$ summands vanishes module $p$, then above equation in $\mathbb{Z}/p\mathbb{Z}$ is in fact
\begin{center}
$0= a_{1}+\ldots+(p-1)a_{p-1}+(p+1)a_{p+1}+...$
\end{center}
or even
\begin{center}
$0= 0$ if all non zero $(z_{i})'s $ are on the form $i=\lambda p$.
\end{center}

In the former case, corresponding to $\nu(n,k,p)$ in Definition \ref{nukp}, we have exactly  $p^{(k-l-1)}$ choices of $a_{1},...,a_{p-1},a_{p+1},...$ and $p^{l}$ choices for $a_{p},a_{2p},...,a_{lp}$. So, there are $p^{k-l-1} \cdot p^{l}=p^{k-1} $ different choices for every solution in $\nu(n,k)$. 

In the latter case, corresponding to $\nu'(n,k,p)$ in Definition \ref{nukp}, $p^{k}$ different solutions are freely chosen for $a_{p},\ldots, a_{kp}$.

Therefore
\begin{center}
$fe_{R}(R^{n})=\displaystyle{\sum_{k=1}^{n}\left(\nu(n,k,p)\cdot p^{k}+\nu'(n,k,p)\cdot p^{k-1}\right)}$
\end{center}

\item[(vi)] If $rk(X)=n$ but $X$ is not free, then $X\cong R^{n-1}\oplus L$ and $\alpha(L)\neq 0$ in $\mathbb{Z}/p\mathbb{Z}$. The equations to compute $fe_{R}(R^{n-1}\oplus L)$ are
\begin{center}
$\left\{\begin{array}{l}
n=\rk(Z_1)+2\rk(Z_2)+\cdots+n\rk(Z_n)$ in $(\mathbb{N}, +) \\
\\
0\neq \alpha(L) =a_{1} + 2a_{2} + \cdots + na_{n}$ in $(\mathbb{Z}/p\mathbb{Z},+) \textit{ where } a_{i}=\alpha(\Det(z_{i}))\\
\end{array}\right.$
\end{center}

Analogous reasoning of $(v)$ gives us to 
\begin{center}
$0\neq \alpha(L) =a_{1} + 2a_{2} + \cdots + (p-1)a_{p-1}+(p+1)a_{p+1}+\ldots$
\end{center}
or
\begin{center}
$0\neq \alpha(K)=0$ having no solution.
\end{center}

Therefore if $X$ is not free of rank $n$ we have
\begin{center}
$fe_{R}(R^{n-1}\oplus L)=\displaystyle{\sum_{k=1}^{n}\nu'(n,k,p)\cdot p^{k-1}}$
\end{center}

\end{enumerate}
\end{proof}

\begin{remark}
Note that if we perform the sum of all computations over elements of $Pic(R)$, then we obatin coherent relationship between our formulae.
\begin{equation*}
\displaystyle{\sum_{L \in Pic(R)} fe_{R}(R^{n-1}\oplus L)}= \displaystyle{\sum_{k=1}^{n}\left(\nu(n,k,p)\cdot p^{k}+\nu'(n,k,p)\cdot p^{k-1}\right)}+ (p-1)\cdot \displaystyle{\sum_{k=1}^{n} \nu'(n,k,p)\cdot p^{k-1}} = 
\end{equation*}
\begin{equation*}
=\displaystyle{\sum_{k=1}^{n} \left(\nu(n,k,p)\cdot p^{k} + p\cdot \nu'(n,k,p)\cdot p^{k-1} \right)}=\displaystyle{\sum_{k=1}^{n} \left(\nu(n,k,p)\cdot p^{k} + \nu'(n,k,p)\cdot p^{k} \right)}=
\end{equation*}
\begin{equation*}
=\displaystyle{\sum_{k=1}^{n} \left(\nu(n,k,p)+ \nu'(n,k,p)\right) \cdot p^{k}}= \displaystyle{\sum_{k=1}^{n} \nu(n,k)\cdot p^{k}}=fe_{R}(n).
\end{equation*}
\end{remark}

\section{Conclusions}
This paper gives a combinatorial approach to a well known problem in systems theory. New (as far as we now) combinatorial numbers $\nu(n,k)$ is introduced. Further study of these combinatorial numbers would be interesting.

A motivation for the study of partitions in monoids is introduced. In particular feedback equivalence problems over product rings translate to partitions over product monoids and feedback equivalence problems over Dedekind domains $R$ translate to partitions and linear equations in $Pic(R)$.


\end{document}